\documentclass[11pt,a4paper]{article}
\usepackage{amsmath,amsthm,amssymb,color,amsfonts}

\theoremstyle{theorem}
\newtheorem{theorem}{Theorem}[section]

\newtheorem{lemma}{Lemma}

\theoremstyle{definition}

\newtheorem{remark}{Remark}

\newtheorem{problem}{Problem}

\begin{document}

\title{\bf Spectral analogues of Moon-Moser's theorem on Hamilton paths in
bipartite graphs}

\date{}

\author{Binlong Li\thanks{Department of Applied Mathematics,
Northwestern Polytechnical University, Xi'an, Shaanxi 710072, P.R.
China. Email: libinlong@mail.nwpu.edu.cn.}~ and Bo
Ning\thanks{Corresponding author. Center for Applied Mathematics,
Tianjin University, Tianjin 300072, P. R. China. Email:
bo.ning@tju.edu.cn}}\maketitle

\begin{center}
\begin{minipage}{130mm}
\small\noindent{\bf Abstract:} In 1962, Erd\H{o}s proved a theorem on
the existence of Hamilton cycles in graphs with given minimum degree and
number of edges. Significantly strengthening in case of balanced bipartite graphs,
Moon and Moser proved a corresponding theorem in 1963. In this paper we establish several
spectral analogues of Moon and Moser's theorem on Hamilton
paths in balanced bipartite graphs and nearly balanced bipartite graphs.
One main ingredient of our proofs is a structural result of its own interest,
involving Hamilton paths in balanced bipartite graphs with given minimum
degree and number of edges.


\end{minipage}
\end{center}

\smallskip

%---------------------
\section{Introduction}
%---------------------
This is a sequel to our previous paper \cite{LN}. In this paper, we are interested
in establishing tight spectral sufficient conditions for Hamilton paths in balanced bipartite
graphs and nearly balanced bipartite graphs.
Throughout this paper, a bipartite graph with the bipartition $\{X,Y\}$
is called \emph{balanced} if $|X|=|Y|$; and is called \emph{nearly balanced} if
$|X|-|Y|=1$ (by the symmetry). A graph $G$ is called \emph{Hamiltonian} if it contains a
spanning cycle, and is called \emph{traceable} if it contains a spanning path.

The topic of Hamiltonicity of graphs has a long history.
In 1961, Ore \cite{O} proved that
every graph on $n$ vertices has a Hamilton cycle if $e(G)>\binom{n-1}{2}+1$.
One year later, Erd\H{o}s \cite{E} generalized Ore's theorem by introducing the minimum degree
of a graph as a new parameter. More precisely, Erd\H{o}s proved that

\begin{theorem}[Erd\H{o}s \cite{E}]
Let $G$ be a graph on $n$ vertices, with minimum degree $\delta(G)$.
If $n/2>\delta(G)\geq k\geq1$, and
\[
e(G)>\max\left\{\binom{n-k}{2}+k^2,\binom{n-\lfloor\frac{n-1}{2}\rfloor}{2}+\left\lfloor\frac{n-1}{2}\right\rfloor^2\right\},
\]
then $G$ is Hamiltonian.
\end{theorem}
Motivated by Erd\H{o}s' work \cite{E},
Moon and Moser \cite{MM} presented some corresponding results
for balanced bipartite graphs. We state one of their theorems as follows,
which is the starting point of our present paper.

\begin{theorem}[Moon and Moser \cite{MM}]\label{ThMoMo}
Let $G$ be a balanced bipartite graph on $2n$ vertices, with minimum degree
$\delta(G)\geq k$, where $1\leq k\leq n/2$. If
$$e(G)>\max\left\{n(n-k)+k^2,n\left(n-\left\lfloor\frac{n}{2}\right\rfloor\right)+{\left\lfloor\frac{n}{2}\right\rfloor}^2\right\},$$
then $G$ is Hamiltonian.
\end{theorem}

Compared with the number of edges of graphs, eigenvalues of graphs
are also very powerful for describing the structure of graphs. There are
several well known examples, such as the spectral proof of
Friendship Theorem \cite{ERS}. For Hamiltonicity
of graphs, early pioneer work include those of van den Heuvel \cite{H},
Krivelevich and Sudakov \cite{KS}, Butler and Chung \cite{BuC}, etc.

Recently, spectral extremal graph theory has rapidly developed,
where extremal properties of graphs are studied
by means of eigenvalues of associated matrices of graphs.
In this area, many beautiful and deep results have been proved,
such as a spectral Tur\'an theorem \cite{N09JCTB}, a spectral
Erd\H{o}s-Stone-Bollob\'as theorem \cite{N09CPC}, a spectral
version of Zarankiewicz problem \cite{N10_1}, spectral sufficient
conditions for paths and cycles \cite{N10_2,YWZ,ZLG}, etc.
For an excellent survey on recent development of spectral
extremal graph theory, we refer the reader to Nikiforov \cite{N}.
In particular, on the topic of Hamiltonicity, Fielder and Nikiforov \cite{FN}
gave spectral analogues of Ore's theorem \cite{O}. More work in this vein
can be found in Zhou \cite{Z}, Lu, Liu and Tian \cite{LLT},
as well as Liu, Shiu and Xue \cite{LSX}. Towards finding spectral
analogues of Erd\H{o}s' theorem, the first attempt was
made by the second author and Ge \cite{NG}, and finally was
completed by the present authors in \cite{LN}. In the meantime,
the authors obtained some spectral analogues of Moon-Moser's theorem
for Hamilton cycles in balanced bipartite graphs \cite{LN}.

One may look for spectral conditions for Hamilton paths in bipartite graphs.
The situation seems a little more complicated. The main reason is
that every traceable bipartite graph must be balanced, or nearly balanced.
That is, there are two situations for us to explore.

In fact, we need to consider the following two Brualdi-Solheid-Tur\'an-type
problems. Here we use $\widehat{G}$ to denote the \emph{quasi-complement} of
a bipartite graph $G$ with the bipartition $\{X,Y\}$,
i.e., one with vertex set $V(\widehat{G})=V(G)$ and for any $x\in X$
and $y\in Y$, $xy\in E(\widehat{G})$ if and only if $xy\notin E(G)$;
and we use $\rho(G)$ and $q(G)$ to denote the spectral radius and signless Laplacian
spectral radius of $G$, respectively.

\begin{problem}\label{ProbBlBi}
Among all non-traceable balanced bipartite graphs $G$ on $2n$ vertices,
with $\delta(G)\geq k$, determine $\max\rho(G), \min\rho(\widehat{G}),
\max q(G)$ and $\min q(\widehat{G})$,
respectively.
\end{problem}

\begin{problem}\label{ProbNeBlBi}
Among all non-traceable nearly balanced bipartite graphs $G$ on
$2n-1$ vertices, with $\delta(G)\geq k$, determine
$\max\rho(G), \min\rho(\widehat{G}), \max q(G)$ and $\min
q(\widehat{G})$, respectively.
\end{problem}

In this paper, we solve the above problems for graphs of
sufficiently large order. The main theorems and related notation
are given in Section 2.

In order to solve these problems, we need to use several spectral inequalities
and convert the original problems into new ones involving the number of edges.
We also use spectral inequalities to characterize the extremal graphs.
In particular, we prove spectral inequalities to compare the
(signless Laplacian) spectral radii of certain types of graphs.
These are given in Section 3.

The proofs of our main theorems also need detailed structural analysis.
We need to use the closure theory of Hamilton cycles in balanced bipartite
graph due to Bondy and Chvat\'al \cite{BoC}. With the help of this theory, we
need to use an analogous theorem for Hamilton paths in balanced bipartite graphs.
We establish a theorem on the existence of a complete bipartite subgraph with
large order in a balanced bipartite graph with sufficiently many edges.
We also prove a theorem on the existence of Hamilton paths in a balanced
bipartite graph with given number of edges. All these structural
lemmas and proofs are given in Section 4.

In Section 5, we prove our main theorems. Finally, in Section 6,
we conclude the paper with some remarks and problems.
\section{Main theorems}
\subsection{Notation}
To describe all extremal graphs in our coming theorems,
we introduce some terminology and notation.
We use $\mathcal{G}_{m,n}$ to denote the set of bipartite graphs
with partition sets of sizes $m$ and $n$. As usual, $K_{m,n}$ denotes the
complete bipartite graph, and we set
$\varPhi_{m,n}=\widehat{K_{m,n}}$. In this paper, when we mention a
bipartite graph, we always fix its partition sets, e.g.,
$\varPhi_{m,n}$ and $\varPhi_{n,m}$ are considered as different
bipartite graphs, unless $m=n$ (although they are both the empty
graphs of order $m+n$).

Let $G_1,G_2$ be two bipartite graphs, with the bipartition $\{X_1,Y_1\}$
and $\{X_2,Y_2\}$, respectively. We use $G_1\sqcup G_2$ to denote the graph
obtained from $G_1\cup G_2$ by adding all possible edges between
$X_1$ and $Y_2$ and all possible edges between $Y_1$ and $X_2$. We
set
$$B_n^k=K_{k,n-k}\sqcup\varPhi_{n-k,k} \mbox{ and }
\mathcal{B}_n^k=\{H\sqcup\varPhi_{n-k,k}: H\in\mathcal{G}_{k,n-k}\}\
(1\leq k\leq n/2).$$
The graphs $B_n^k$ play a crucial role in the proofs of results in \cite{LN}.
Notice that $B_n^k$ is the graph in $\mathcal{B}_n^k$ with the largest
number of edges. We remark that for any (spanning) subgraph $G$ of $B_n^k$,
$\rho(\widehat{G})=\rho(\widehat{B_n^k})$
($q(\widehat{G})=q(\widehat{B_n^k})$) if and only if
$G\in\mathcal{B}_n^k$.

We define some classes of graphs as follows:
\begin{align*}
\begin{array}{ll}
Q_n^k=K_{k,n-k-1}\sqcup\varPhi_{n-k,k+1}  & (0\leq k\leq(n-1)/2), \\
R_n^k=K_{k,k}\cup K_{n-k,n-k}             & (1\leq k\leq n/2), \\
S_n^k=K_{k,n-k-1}\sqcup\varPhi_{n-k,k}    & (1\leq k\leq(n-1)/2),\\
\mathcal{S}_n^k=\{H\sqcup\varPhi_{k,n-k}: H\in\mathcal{G}_{n-k-1,k}\} & (1\leq k\leq(n-1)/2),\\
T_n^k=K_{k,n-k-1}\sqcup\varPhi_{n-k-1,k+1}& (0\leq k\leq n/2-1),\\
\mathcal{T}_n^k=\{H\sqcup\varPhi_{k+1,n-k-1}:
H\in\mathcal{G}_{n-k-1,k}\} & (0\leq k\leq n/2-1).
\end{array}
\end{align*}

Additionally, let $\varGamma_n^0=K_{n-2,n}\cup K_{1,0}$ and let
\emph{$\L$} be the graph in Fig. 1.

\begin{center}
\begin{picture}(80,60)\label{fi1}
\thicklines \multiput(20,30)(20,0){3}{\put(0,0){\circle*{4}}}
\put(20,30){\line(1,0){40}}
\multiput(40,10)(0,40){2}{\multiput(0,0)(20,0){2}{\circle*{4}}
\put(0,0){\line(1,0){20}}} \put(20,30){\line(1,1){20}}
\put(20,30){\line(1,-1){20}}
\end{picture}

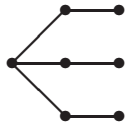
\small Fig. 1. The graph $\emph{\L}$.
\end{center}

Note that $S_n^k$ is the graph in $\mathcal{S}_n^k$ with the largest
number of edges, and $T_n^k$ is the graph in $\mathcal{T}_n^k$ with
the largest number of edges. Similarly, we remark that for any
(spanning) subgraph $G$ of $S_n^k$,
$\rho(\widehat{G})=\rho(\widehat{S_n^k})$
($q(\widehat{G})=q(\widehat{S_n^k})$) if and only if
$G\in\mathcal{S}_n^k$; and for any (spanning) subgraph $G$ of
$T_n^k$, $\rho(\widehat{G})=\rho(\widehat{T_n^k})$
($q(\widehat{G})=q(\widehat{T_n^k})$) if and only if
$G\in\mathcal{T}_n^k$.
\subsection{Main results}
In this subsection, we state all our main theorems. Since we consider
the classes of balanced bipartite graphs and nearly balanced bipartite
graphs, and for each class of graphs, we consider sufficient conditions
in terms of (signless Laplacian) spectral radii of graphs or the complements,
we obtain eight theorems as follows.

For balanced bipartite graphs, we have
\begin{theorem}\label{ThrhoBG}
Let $G$ be a balanced bipartite graph on $2n$ vertices, with minimum
degree $\delta(G)\geq k$, where $k\geq 0$ and $n\geq(k+2)^2$.\\
(1) If $k\neq 1$ and $\rho(G)\geq\rho(Q^k_n)$,
then $G$ is traceable unless $G=Q^k_n$.\\
(2) If $k=1$ and $\rho(G)\geq\rho(R^1_n)$, then $G$ is traceable
unless $G=R^1_n$.
\end{theorem}

\begin{theorem}\label{ThqBG}
Let $G$ be a balanced bipartite graph on $2n$ vertices, with minimum
degree $\delta(G)\geq k$, where $k\geq 0$ and $n\geq(k+2)^2$.
If $q(G)\geq q(Q^k_n)$, then $G$ is traceable unless $G=Q^k_n$.
\end{theorem}

\begin{theorem}\label{ThrhoBGC}
Let $G$ be a balanced bipartite graph on $2n$ vertices, with minimum
degree $\delta(G)\geq k$, where $k\geq 0$ and $n\geq 2k$.\\
(1) If $k\geq 1$ and $\rho(\widehat{G})\leq\rho(\widehat{R^k_n})$,
then $G$ is traceable unless $G=R^k_n$. \\
(2) If $k=0$ and $\rho(\widehat{G})\geq\rho(\widehat{Q^0_n})$, then
$G$ is traceable unless $G=Q_n^0$.
\end{theorem}

\begin{theorem}\label{ThqBGC}
Let $G$ be a balanced bipartite graph on $2n$ vertices. If
$q(\widehat{G})\leq n$, then $G$ is traceable unless $G\in\{R_n^k:
1\leq k\leq\lfloor n/2\rfloor\}$.
\end{theorem}

\begin{remark}
Our Theorem 2.1 generalizes Theorem 2.10 in \cite{LSX} due to Liu et al.
\end{remark}

For nearly balanced bipartite graphs, we have
\begin{theorem}\label{ThrhoNBG}
Let $G$ be a nearly balanced bipartite graph on $2n-1$ vertices, with
minimum degree $\delta(G)\geq k$, where $k\geq 0$ and $n\geq(k+1)^2$.\\
(1) If $k\geq 1$ and $\rho(G)\geq\rho(S^k_n)$, then $G$ is traceable
unless $G=S^k_n$. \\
(2) If $k=0$ and $\rho(G)\geq\rho(T^0_n)$, then $G$ is traceable
unless $G=T_n^0$.
\end{theorem}

\begin{theorem}\label{ThqNBG}
Let $G$ be a nearly balanced bipartite graph on $2n-1$ vertices, with
minimum degree $\delta(G)\geq k$, where $k\geq 0$ and $n\geq(k+1)^2$.\\
(1) If $k\geq 1$ and $q(G)\geq q(S^k_n)$, then $G$ is traceable
unless $G=S^k_n$.\\
(2) If $k=0$ and $q(G)\geq q(S^1_n)$, then $G$ is traceable unless
$G=S_n^1$.
\end{theorem}

\begin{theorem}\label{ThrhoNBGC}
Let $G$ be a nearly balanced bipartite graph on $2n-1$ vertices, with
minimum degree $\delta(G)\geq k$, where $k\geq 0$ and $n\geq 2k+1$. \\
(1) If $k\geq 1$ and $\rho(\widehat{G})\leq\rho(\widehat{S_n^k})$,
then $G$ is traceable unless $G\in\mathcal{S}_n^k$.\\
(2) If $k=0$ and $\rho(\widehat{G})\leq\rho(\widehat{T^0_n})$, then
$G$ is traceable unless $G\in\mathcal{S}_n^1\cup\{T_n^0\}$.
\end{theorem}

\begin{theorem}\label{ThqNBGC}
Let $G$ be a nearly balanced bipartite graph on $2n-1$ vertices. If
$q(\widehat{G})\leq n$, then $G$ is traceable unless
$G\in(\bigcup_{k=1}^{\lfloor
(n-1)/2\rfloor}\mathcal{S}_n^k)\cup(\bigcup_{k=0}^{\lfloor
n/2\rfloor-1}\mathcal{T}_n^k)$, or $n=4$ and $G={\L}$.
\end{theorem}

\section{Spectral inequalities}
%---------------------------------

We will use the following spectral inequalities for graphs and
bipartite graphs, respectively. The first theorem is a direct
corollary of a result of Nosal \cite{No}. (See also \cite{BFP}.)

\begin{theorem}[Nosal \cite{No}, Bhattacharya, Friedland and Peled \cite{BFP}]\label{ThBhFrPe}
Let $G$ be a bipartite graph. Then
$$\rho(G)\leq \sqrt{e(G)}.$$
\end{theorem}
The next theorem has been proved in \cite{LN}, with the help
of a result due to Feng and Yu \cite[Lemma~2.4]{FY}, which
can be traced back to Merris \cite{M}.
\begin{theorem}[Li and Ning \cite{LN}]\label{ThLiNi5}
Let $G$ be a balanced bipartite graph on $2n$ vertices. Then
$$q(G)\leq \frac{e(G)}{n}+n.$$
\end{theorem}
The following two theorems can be proved similarly as Lemma 2.1 in
\cite{BZ} and Theorem 2 in \cite{AM}, respectively. We omit the proofs.

\begin{theorem}\label{ThMinrhoG}
Let $G$ be a graph with non-empty edge set. Then
$$\rho(G)\geq\min\{\sqrt{d(u)d(v)}: uv\in E(G)\}.$$
Moreover, if $G$ is connected, then equality holds if and only if
$G$ is regular or semi-regular bipartite.
\end{theorem}

\begin{theorem}\label{ThMinqG}
Let $G$ be a graph with non-empty edge set. Then
$$q(G)\geq\min\{d(u)+d(v): uv\in E(G)\}.$$
Moreover, if $G$ is connected, then the equality holds if and only if
$G$ is regular or semi-regular bipartite.
\end{theorem}

\begin{lemma}\label{LeCompare} For $k\geq 1$, $n\geq 2k+1$, we have
\begin{align*}
& \rho(Q_n^k)>\rho(K_{n,n-k-1})=\sqrt{n(n-k-1)}; && \rho(S_n^k)>\rho(K_{n,n-k-1})=\sqrt{n(n-k-1)};\\
& q(Q_n^k)>q(K_{n,n-k-1})=2n-k-1;               && q(S_n^k)>q(K_{n,n-k-1})=2n-k-1;\\
& \rho(\widehat{R_n^k})=\rho(K_{k,n-k})=\sqrt{k(n-k)};   && \rho(\widehat{S_n^k})=\rho(K_{n-k,k})=\sqrt{k(n-k)};\\
& q(\widehat{R_n^k})=q(K_{k,n-k})=n;           &&
q(\widehat{S_n^k})=q(K_{n-k,k})=q(\widehat{T_n^k})=q(K_{n-k-1,k+1})=n.
\end{align*}
\end{lemma}
\begin{proof}
Since $K_{n,n-k-1}$ is a proper subgraph of $Q_n^k$ or $S_n^k$,
the first four inequalities follow from the Perron-Frobenius Theorem.
The others can be checked easily.
\end{proof}

\begin{lemma}\label{LeQn1Sn1}
(1) For $n\geq 3$,
$\rho(S_n^1)<\rho(Q_n^1)\leq\rho(R_n^1)=\rho(T_n^0)=n-1$, where the second
inequality becomes equality only if $n=3$.\\
(2)~For $n\geq 3$,
$2n-1=q(Q_n^0)>q(Q_n^1)>q(R_n^1)=2n-2$.
\end{lemma}

\begin{proof}
(1) First, note that $S_n^1$ is a proper subgraph of $Q_n^1$. Thus,
$\rho(S_n^1)<\rho(Q_n^1)$.

Next, we show that $\rho(Q_n^1)\leq n-1$. Recall that $Q_n^1=K_{1,n-2}\sqcup\varPhi_{n-1,2}$.
Let $\{X_1,Y_1\}$ be the bipartition of $K_{1,n-2}$, where $|X_1|=1,|Y_1|=n-2$. Let $\{X_2,Y_2\}$
be the bipartition of $\varPhi_{n-1,2}$, where $|X_2|=n-1$, $|Y_2|=2$.
Set $\rho=\rho(Q_n^1)$. Let $X=(x_1,x_2,\ldots,x_n)$ be a positive unit
eigenvector of $Q_n^1$ corresponding to $\rho$.
Since any pair of vertices in the same partite set, say $v_1,v_2$, have the same neighborhood,
we know $x_{v_1}=x_{v_2}$.
Thus, we can assume that
\begin{align*}
& x:=x_v, v\in X_1;\\
& y:=x_v,v\in Y_1;\\
& z:=x_v, v\in X_2;\\
& t:=x_v, v\in Y_2.
\end{align*}
The eigenvalue equations can be reduced to the following four ones:
\begin{align}
&\rho x=(n-2)y+2t,\label{eq1}\\
&\rho y=x+(n-1)z,\label{eq2}\\
&\rho z=(n-2)y,\label{eq3}\\
&\rho t=x\label{eq4}.
\end{align}
Multiplying the two sides of (\ref{eq2}) by $\rho$, and putting (\ref{eq3}) into it, we have
\begin{align*}
{\rho}^2y=\rho x+(n-1)(n-2)y,
\end{align*}
that is,
\begin{align}
[{\rho}^2-(n-1)(n-2)]y=\rho x.\label{eq5}
\end{align}
Similarly, multiplying the two sides of (\ref{eq1}) by $\rho$, and eliminating $t$, we obtain
\begin{align}
({\rho}^2-2)x=(n-2)\rho y.\label{eq6}
\end{align}
Combining (\ref{eq5}) and (\ref{eq6}), and cancelling $xy$ yields
\begin{align}
\rho^4-(n^2-2n+2)\rho^2+2(n-1)(n-2)=0.\label{eq7}
\end{align}
By solving Equation (\ref{eq7}), we obtain
\begin{align*}
\rho^2=\frac{(n^2-2n+2)+\sqrt{(n^2-2n+2)^2-8(n-1)(n-2)}}{2}.
\end{align*}
By simple algebra, we get $\rho^2< (n-1)^2$ when $n\geq 4$ and $\rho=n-1$
when $n=3$.

(2) Since $Q_n^1$ contains $K_{n,n-2}$ as its proper subgraph, from the Perron-Frobenius Theorem,
we can see $q(Q_n^1)>q(K_{n,n-2})=2n-2=q(R_n^1)$. On the other hand, by Theorem \ref{ThLiNi5},
we have
\[
q(Q_n^1)\leq \frac{e(Q_n^1)}{n}+n=\frac{n(n-2)+2}{n}+n=2n-2+\frac{2}{n}<2n-1
\]
when $n\geq 3$. This proves the statement (2).
\end{proof}

%----------------------
\section{Structural lemmas}
%----------------------
In this section, we state some known structural theorems and
prove some new ones.

The first tool we need is the closure theory of Hamilton cycles in
balanced bipartite graphs introduced by Bondy and Chv\'{a}tal \cite{BoC}.
Let $G$ be a balanced bipartite graph on $2n$ vertices. The
\emph{bipartite closure} (or briefly, \emph{B-closure}) of $G$,
denoted by $cl_B(G)$, is the graph obtained from $G$ by recursively
joining pairs of nonadjacent vertices in different partition sets
whose degree sum is at least $n+1$ until no such pair remains. A
balanced bipartite graph $G$ on $2n$ vertices is \emph{B-closed} if
$G=cl_B(G)$, i.e., if every two nonadjacent vertices in distinct
partition sets of $G$ have degree sum at most $n$.

\begin{theorem}[Bondy and Chv\'{a}tal \cite{BoC}]\label{ThBoCh}
A balanced bipartite graph $G$ is Hamiltonian if and only if
$cl_B(G)$ is Hamiltonian.
\end{theorem}

\begin{lemma}\label{LeClosed}\footnote{This result may have appeared in
some early reference, but we could not find any. We include its short proof
here to keep our paper self-contained.}
A balanced bipartite graph $G$ is traceable if and only if $cl_B(G)$
is traceable.
\end{lemma}

\begin{proof}
Clearly $G$ being traceable implies that $cl_B(G)$ being traceable. Now
we assume that $cl_B(G)$ is traceable. If $cl_B(G)$ is Hamiltonian,
then $G$ is Hamiltonian by Theorem \ref{ThBoCh}. Now we assume that
$cl_B(G)$ has a Hamilton path $P$ but no Hamilton cycle. Let $x,y$
be the two end-vertices of $P$. Then $xy\notin E(cl_B(G))$.

Let $G'=G+xy$. Then $cl_B(G)+xy \subseteq cl_B(G')$. Thus, $cl_B(G')$ is Hamiltonian,
and $G'$ is Hamiltonian by Theorem \ref{ThBoCh}. So, $G$ is traceable.
\end{proof}

We need two theorems proved in \cite{LN}.
\begin{theorem}[Li and Ning \cite{LN}]\label{ThLiNi3}
Let $G$ be a B-closed balanced bipartite graph on $2n$ vertices. If
$n\geq 2k+1$ for some $k\geq 1$ and $$e(G)>n(n-k-1)+(k+1)^2,$$ then
$G$ contains a complete bipartite subgraph of order $2n-k$.
Furthermore, if $\delta(G)\geq k$, then $K_{n,n-k}\subseteq G$.
\end{theorem}

\begin{theorem}[Li and Ning \cite{LN}]\label{ThLiNi4}
Let $G$ be a balanced bipartite graph on $2n$ vertices. If
$\delta(G)\geq k\geq 1$, $n\geq 2k+1$ and
$$e(G)>n(n-k-1)+(k+1)^2,$$ then $G$ is Hamiltonian unless $G\subseteq B_n^k$.
\end{theorem}
Using the above two theorems, we prove the following corresponding lemmas
for the existence of Hamilton paths and complete bipartite subgraphs
in balanced bipartite graphs, respectively.
\begin{lemma}\label{LeSubgraph}
Let $G$ be a B-closed balanced bipartite graph on $2n$ vertices. If
$n\geq 2k+3$ for some $k\geq 1$ and $$e(G)>n(n-k-2)+(k+2)^2,$$ then
$G$ contains a complete bipartite subgraph on $2n-k-1$ vertices.
Furthermore, if $\delta(G)\geq k$, then $K_{n,n-k-1}\subseteq G$, or
$k=1$ and $K_{n-1,n-1}\subseteq G$.
\end{lemma}

\begin{proof}
The existence of a complete bipartite subgraph on $2n-k-1$ vertices
can be deduced from Theorem \ref{ThLiNi3}. Let $X,Y$ be the partition sets
of $G$, and $X'\subseteq X$, $Y'\subset Y$ such that $G[X'\cup
Y']=K_{s,t}$, where $s+t\geq 2n-k-1$ and $s\geq t$. We choose $s,t$
such that $s$ is as large as possible.

Now suppose that $\delta(G)\geq k$. If $K_{n,n-k-1}\not\subseteq G$,
then $n-k\leq t\leq s\leq n-1$. Note that every vertex in
$X\backslash X'$ has degree at least $k$ and every vertex in $Y'$
has degree at least $s$. If $s+k\geq n+1$, then every vertex in
$X\backslash X'$ and every vertex in $Y'$ are adjacent, and
$K_{n,n-k}\subseteq G$, a contradiction. This implies that
$s+k\leq n$, i.e., $s\leq n-k$. Hence we have $s=t=n-k$. Recall that
$s+t\geq 2n-k-1$. We have $k=1$ and $s=t=n-1$. Thus
$K_{n-1,n-1}\subseteq G$.
\end{proof}

\begin{lemma}\label{LeTraceable}
Let $G$ be a balanced bipartite graph on $2n$ vertices. If
$\delta(G)\geq k\geq 1$, $n\geq 2k+3$ and
$$e(G)>n(n-k-2)+(k+2)^2,$$ then $G$ is traceable unless $G\subseteq
Q_n^k$, or $k=1$ and $G\subseteq R_n^1$.
\end{lemma}

\begin{proof}
Let $G'=cl_B(G)$. If $G'$ is traceable, then so is $G$ by Lemma
\ref{LeClosed}. Now we assume that $G'$ is not traceable. We first
deal with the case $k=1$. Note that $\delta(G')\geq\delta(G)$ and
$e(G')\geq e(G)$. By Lemma \ref{LeSubgraph}, either
$K_{n,n-2}\subseteq G'$ or $K_{n-1,n-1}\subseteq G'$. Recall that
$\delta(G')\geq 1$. It is easy to check the only
non-traceable balanced bipartite graphs of order $2n$ without
isolated vertices containing $K_{n,n-2}$ or $K_{n-1,n-1}$ are
$Q_n^1$ and $R_n^1$, respectively. Thus $G'=Q_n^1$ or $R_n^1$, and this implies
that $G\subseteq Q_n^1$ or $G\subseteq R_n^1$.

Now assume that $k\geq 2$. By Lemma \ref{LeSubgraph},
$K_{n,n-k-1}\subseteq G'$. Let $t$ be the largest integer such that
$K_{n,t}\subseteq G$. Clearly $n-k-1\leq t<n$. Let $X,Y$ be the
partition sets of $G$, and $Y'\subset Y$ such that $G[X\cup
Y']=K_{n,t}$.

We first claim that $t=n-k-1$. If $t\geq n-k+1$, then every vertex
of $X$ has degree at least $n-k+1$ in $G'$ and every vertex in $Y$
has degree at least $k$ in $G'$, implying that $G'$ is complete
bipartite. Thus $G'$ is traceable, a contradiction. Suppose now that
$t=n-k$. If some vertex in $Y\backslash Y'$ has degree at least
$k+1$ in $G'$, then it will be adjacent to every vertex in $X$ in
$G'$, a contradiction. So we conclude that every vertex in
$Y\backslash Y'$ has degree exactly $k$. If a vertex $x\in X$ is
adjacent to some vertex in $Y\backslash Y'$, then $d_{G'}(x)\geq
n-k+1$ and $x$ will be adjacent to every vertex in $Y\backslash Y'$.
This implies that all the vertices in $Y\backslash Y'$ are adjacent
to $k$ common vertices in $X$, i.e., $G'=B_n^k$. Note that $B_n^k$
is traceable, a contradiction. Thus $t=n-k-1$, as we claimed.

Next we show that every vertex of $Y\backslash Y'$ has degree exactly
$k$. Suppose that there is a vertex $y\in Y\backslash Y'$ which has
degree at least $k+1$ in $G'$. If $d_{G'}(y)\geq k+2$, then since
$d_{G'}(x)\geq n-k-1$ for every $x\in X$, $y$ will be adjacent to
every vertex of $X$, a contradiction. So we have $d_{G'}(y)=k+1$.
Let $X'$ be the set of $n-k-1$ vertices in $X$ nonadjacent to $y$.
Then for every $x\in X'$, $x$ is nonadjacent to any vertex of
$Y\backslash Y'$; otherwise $d_{G'}(x)\geq n-k$, implying that $xy\in
E(G)$. Now consider the subgraph $H=G'[X\backslash X',Y\backslash
Y']$. Note that for every $y'\in Y\backslash Y'$, $d_H(y')\geq k$
and for every $x'\in X\backslash X'$, $d_H(x')\geq 1$. If every
vertex in $X\backslash X'$ has degree at least 2 in $H$, then
$cl_B(H)$ is complete and bipartite, implying that $H$ is traceable; if there
is a vertex, say $x$ in $X\backslash X'$, with degree 1 in $H$,
i.e., $x$ has only one neighbor $y$ in $H$, then $H-\{x,y\}$ is
complete and bipartite, also implying that $H$ is traceable. Note that
$G'[X,Y']$ is complete. So $G'$ is traceable, a contradiction. Thus we
conclude that every vertex of $Y\backslash Y'$ has degree exactly
$k$.

Let $x$ be an arbitrary vertex in $X$. If $x$ is adjacent to at
least two vertices in $Y\backslash Y'$, then $d(x)\geq n-k+1$,
implying that $x$ is adjacent to all vertices in $Y\backslash Y'$. Thus
we conclude that every vertex in $X$ is adjacent to either no
vertices, or only one vertex, or all vertices in $Y\backslash Y'$.
We call the vertex $x$ a \emph{simple} (\emph{frontier}, \emph{saturated}, resp.)
vertex if $x$ is adjacent to no (one, every, resp.) vertex in
$Y\backslash Y'$.

If every vertex in $Y\backslash Y'$ is adjacent to at least two
frontier vertices, then we can take $k+1$ vertex-disjoint $P_3$'s
such that every vertex in $Y\backslash Y'$ is the center of a $P_3$.
Since $G'[X,Y']$ is complete and bipartite, it is easy to check that
$G'$ is traceable, a contradiction. If every vertex in $Y\backslash
Y'$ is adjacent to exactly one frontier vertex, implying that there
are $k-1$ saturated vertices. (Note that every vertex in $Y\backslash
Y'$ is adjacent to the same number of frontier vertices.) In this
case, there are $k-1$ vertex-disjoint $P_3$'s with the centers in
$Y\backslash Y'$ and two additional independent edges incident to
vertices in $Y\backslash Y'$. Since $G'[X,Y']$ is complete and
bipartite, it is easy to check that $G'$ is traceable, a
contradiction.

Now assume that there are no frontier vertices. Thus every vertex
in $Y\backslash Y'$ is adjacent to (the common) $k$ saturated
vertices. In this case $G'=Q_n^k$ and $G\subseteq Q_n^k$.
\end{proof}

Finally, we recall two theorems proved in \cite{LN}.
\begin{theorem}[Li and Ning \cite{LN}]\label{ThLiNi}
Let $G$ be a balanced bipartite graph on $2n$ vertices, with minimum
degree $\delta(G)\geq k\geq 1$. \\
(1) If $n\geq (k+1)^2$ and $\rho(G)\geq\rho(B^k_n)$, then $G$ is
Hamiltonian unless $G=B^k_n$.\\
(2) If $n\geq (k+1)^2$ and $q(G)\geq q(B^k_n)$, then $G$ is
Hamiltonian unless $G=B^k_n$.\\
(3) If $n\geq 2k$ and $\rho(\widehat{G})\leq\rho(\widehat{B^k_n})$,
then $G$ is Hamiltonian unless $G\in\mathcal{B}_n^k$, or $k=2$,
$n=4$ and $G=L_1$ or $L_2$ (see Fig. 2).
\end{theorem}
\begin{center}
\begin{picture}(200,80)
\thicklines

\put(0,0){\multiput(20,30)(20,0){4}{\circle*{4}}
\multiput(20,70)(20,0){4}{\circle*{4}}
\multiput(80,30)(-20,0){3}{\line(0,1){40}}
\put(20,30){\line(1,2){20}} \put(20,30){\line(1,1){40}}
\put(20,30){\line(3,2){60}} \put(20,70){\line(1,-2){20}}
\put(20,70){\line(1,-1){40}} \put(20,70){\line(3,-2){60}}
\put(45,10){$L_1$} }

\put(90,0){\multiput(20,30)(20,0){4}{\circle*{4}}
\multiput(20,70)(20,0){4}{\circle*{4}}
\multiput(80,30)(-20,0){4}{\line(0,1){40}}
\put(20,30){\line(1,2){20}} \put(20,30){\line(1,1){40}}
\put(20,30){\line(3,2){60}} \put(20,70){\line(1,-2){20}}
\put(20,70){\line(1,-1){40}} \put(20,70){\line(3,-2){60}}
\put(45,10){$L_2$} }

\end{picture}

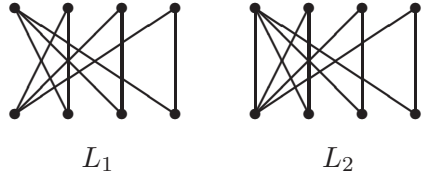
\small Fig. 2. The graphs $L_1$ and $L_2$.
\end{center}

\begin{theorem}[Li and Ning \cite{LN}]\label{ThLiNi'}
Let $G$ be a balanced bipartite graph on $2n$ vertices. If
$q(\widehat{G})\leq n$, then $G$ is Hamiltonian unless
$G\in\bigcup_{k=1}^{\lfloor n/2\rfloor}\mathcal{B}_n^k$, or $n=4$
and $G=L_1$ or $L_2$ (see Fig. 2).
\end{theorem}
%---------------------------------

%-------------------------------
\section{Proofs}
%-------------------------------
In this section, we prove our main theorems.

\bigskip\noindent\textbf{Proof of Theorem \ref{ThrhoBG}.}
Suppose that $G$ is not traceable. If $k\geq 1$, then by Lemmas
\ref{LeCompare}, \ref{LeQn1Sn1}(1) and Theorem \ref{ThBhFrPe},
$$\sqrt{e(G)}\geq\rho(G)\geq\rho(Q_n^k)>\sqrt{n(n-k-1)}.$$
Thus, we have
\begin{align*}
e(G)>n(n-k-1)\geq n(n-k-2)+(k+2)^2
\end{align*}
when $n\geq(k+2)^2$. Since $n\geq(k+2)^2>2k+3$, by Lemma \ref{LeTraceable}, $G\subseteq
Q_n^k$ or $k=1$ and $G\subseteq R_n^1$. If $k\geq 2$, then
$G\subseteq Q_n^k$. But if $G\varsubsetneq Q_n^k$, then
$\rho(G)<\rho(Q_n^k)$, a contradiction. Thus $G=Q_n^k$. If $k=1$,
then $G\subseteq Q_n^1$ or $G\subseteq R_n^1$. But if $G\subseteq
Q_n^1$ or $G\varsubsetneq R_n^1$, then by Lemma \ref{LeQn1Sn1}(1),
we get $\rho(G)<\rho(R_n^1)$, a contradiction. Thus $G=R_n^1$.

Now assume that $k=0$. If $G$ has no isolated vertex, i.e.,
$\delta(G)\geq 1$, then by the above analysis,
$$\rho(G)\leq\rho(R_n^1)=n-1<\rho(Q_n^0)=\sqrt{n(n-1)},$$
a contradiction. Thus $G$ has an isolated vertex and $G\subseteq
Q_n^0$. But if $G\varsubsetneq Q_n^0$, then $\rho(G)<\rho(Q_n^0)$, a
contradiction. Thus $G=Q_n^0$.
{\hfill$\Box$}

\bigskip\noindent\textbf{Proof of Theorem \ref{ThqBG}.}
Suppose that $G$ is not traceable. If $k\geq 1$, then by Lemma
\ref{LeCompare} and Theorem \ref{ThLiNi5}, we have
$$\frac{e(G)}{n}+n\geq q(G)\geq q(Q_n^k)>2n-k-1.$$
Thus, we have
\begin{align*}
e(G)>n(n-k-1)\geq n(n-k-2)+(k+2)^2
\end{align*}
when $n\geq(k+2)^2$. By Lemma \ref{LeTraceable}, $G\subseteq
Q_n^k$ or $k=1$ and $G\subseteq R_n^1$. If $k\geq 2$, then
$G\subseteq Q_n^k$. But if $G\varsubsetneq Q_n^k$, then $q(G)<q(Q_n^k)$, a
contradiction. Thus $G=Q_n^k$. If $k=1$, then $G\subseteq Q_n^1$ or
$G\subseteq R_n^1$. But if $G\varsubsetneq Q_n^1$ or $G\subseteq R_n^1$,
then by Lemma \ref{LeQn1Sn1}(2), we obtain $q(G)<q(Q_n^1)$,
a contradiction. Thus $G=Q_n^1$.

Now assume that $k=0$. If $G$ has no isolated vertex, i.e.,
$\delta(G)\geq 1$, then by the analysis above and Lemma \ref{LeQn1Sn1}(2), we obtain
$$q(G)\leq q(Q_n^1)<q(Q_n^0)=2n-1,$$
a contradiction. Thus $G$ has an isolated vertex and $G\subseteq
Q_n^0$. But if $G\varsubsetneq Q_n^0$, then $q(G)<q(Q_n^0)$, a
contradiction. Thus $G=Q_n^0$.
{\hfill$\Box$}

\bigskip\noindent\textbf{Proof of Theorem \ref{ThrhoBGC}.}
Suppose that $G$ is not traceable. Then $G$ is not hamiltonian. If
$k\geq 1$, then
$$\rho(\widehat{G})\leq\rho(\widehat{R_n^k})=\rho(\widehat{B_n^k})=\sqrt{k(n-k)}.$$ By
Theorem \ref{ThLiNi}, $G\in\mathcal{B}_n^k$, or $k=2$, $n=4$ and
$G=L_1$ or $L_2$. But if $G\in\mathcal{B}_n^k\backslash\{R_n^k\}$,
or $G=L_1$ or $L_2$, then $G$ is traceable, a contradiction. Thus we
conclude $G=R_n^k$.

Now assume that $k=0$. If $G$ has no isolated vertex, i.e.,
$\delta(G)\geq 1$, then by the above analysis,
$$\rho(\widehat{G})\leq\rho(\widehat{R_n^1})=\sqrt{n-1}<\rho(\widehat{Q_n^0})=\sqrt{n},$$
a contradiction. This implies that $G$ has an isolated vertex and
$G\subseteq Q_n^0$. But if $G\varsubsetneq Q_n^0$, then
$\rho(\widehat{G})>\rho(\widehat{Q_n^0})$, a contradiction. Thus
$G=Q_n^0$.
{\hfill$\Box$}

\bigskip\noindent\textbf{Proof of Theorem \ref{ThqBGC}.}
Suppose that $G$ is not traceable. Then $G$ is not hamiltonian. By
Theorem \ref{ThLiNi'}, $G\in\bigcup_{k=1}^{\lfloor
n/2\rfloor}\mathcal{B}_n^k$, or $n=4$ and $G=L_1$ or $L_2$. But if
$G\in\bigcup_{k=1}^{\lfloor
n/2\rfloor}(\mathcal{B}_n^k\backslash\{R_n^k\})$, or $G=L_1$ or
$L_2$, then $G$ is traceable, a contradiction. Thus we conclude
$G\in\{R_n^k: 1\leq k\leq\lfloor n/2\rfloor\}$.
{\hfill$\Box$}

\bigskip\noindent\textbf{Proof of Theorem \ref{ThrhoNBG}.}
Let $\{X,Y\}$ be the partition of $V(G)$ such that $|X|=n-1$ and
$|Y|=n$. Let $G'$ be the graph obtained from $G$ by adding one new
vertex $x'$ and connecting $x'$ to every vertex in $Y$ by an edge.
Clearly $G$ is traceable if and only if $G'$ is Hamiltonian.

If $k\geq 1$, then by Lemma \ref{LeCompare} and Theorem
\ref{ThBhFrPe},
$$\sqrt{e(G)}\geq\rho(G)\geq\rho(S_n^k)>\sqrt{n(n-k-1)}.$$
Thus, we have
\begin{align*}
e(G)>n(n-k-1)\geq n(n-k-2)+(k+1)^2
\end{align*}
when $n\geq(k+1)^2$. This implies that $e(G')>n(n-k-1)+(k+1)^2$.
Note that $\delta(G')\geq\delta(G)\geq k$. By Theorem \ref{ThLiNi4},
$G'$ is Hamiltonian or $G'\subseteq B_n^k$. Thus, $G$ is traceable or
$G\subseteq S_n^k$ or $G\subseteq T_n^{k-1}$. But if $G\varsubsetneq
S_n^k$, then $\rho(G)<\rho(S_n^k)$; if $G\subseteq T_n^{k-1}$, then
$\delta(G)\leq k-1$. Thus $G=S_n^k$.

Now assume that $k=0$. If $G$ has no isolated vertex, then
$\delta(G)\geq 1$. If $n=2$, then clearly $G$ is traceable. So we
may assume that $n\geq 3$. By the above analysis, and by Lemma
\ref{LeQn1Sn1}(1),
$$\rho(G)\leq\rho(S_n^1)<\rho(T_n^0)=n-1,$$ a contradiction. This
implies that $G$ has an isolated vertex, and $G\varsubsetneq T_n^0$ or
$G\subseteq\varGamma_n^0$. But if $G\varsubsetneq T_n^0$ or
$G\subseteq\varGamma_n^0$, then $\rho(G)<\rho(T_n^0)$, a
contradiction. Thus $G=T_n^0$.
{\hfill$\Box$}

\bigskip\noindent\textbf{Proof of Theorem \ref{ThqNBG}.}
Let $G'$ be defined as in the proof of Theorem \ref{ThrhoNBG}. If
$k\geq 1$, then by Lemma \ref{LeCompare} and Theorem \ref{ThLiNi5}, we have
$$\frac{e(G)}{n}+n\geq q(G)>2n-k-1.$$
Note that here we consider $G$ as a balanced bipartite graph having an isolated vertex. Thus
\begin{align*}
e(G)>n(n-k-1)\geq n(n-k-2)+(k+1)^2
\end{align*}
when $n\geq(k+1)^2$. This implies that $e(G')>n(n-k-1)+(k+1)^2$.
Note that $\delta(G')\geq\delta(G)\geq k$. By Theorem \ref{ThLiNi4},
$G'$ is Hamiltonian or $G'\subseteq B_n^k$. Thus $G$ is traceable or
$G\subseteq S_n^k$ or $G\subseteq T_n^{k-1}$. But if $G\varsubsetneq S_n^k$,
then $q(G)<q(S_n^k)$; if $G\subseteq T_n^{k-1}$, then
$\delta(G)\leq k-1$. Thus $G=S_n^k$.

Now assume that $k=0$. If $G$ has an isolated vertex, then
$G\subseteq T_n^0$ or $G\subseteq\varGamma_n^0$. But if $G\varsubsetneq
T_n^0$ or $G\subseteq\varGamma_n^0$, then $q(G)<q(S_n^1)$, a
contradiction. Here notice that $q(T_n^0)=q(\varGamma_n^0)=2n-2$,
and $K_{n,n-2}\varsubsetneq S_n^1$. So we assume that
$G$ has no isolated vertex, i.e., $\delta(G)\geq 1$. By the above
analysis, $G$ is traceable unless $G=S_n^1$.
{\hfill$\Box$}

\bigskip\noindent\textbf{Proof of Theorem \ref{ThrhoNBGC}.}
We suppose first that $k\geq 1$. Let $G'$ be defined as in the proof
of Theorem \ref{ThrhoNBG}. Note that $\delta(G')\geq\delta(G)\geq k$
and
$\rho(\widehat{G'})=\rho(\widehat{G})\leq\rho(\widehat{S_n^k})=\rho(\widehat{B_n^k})$.
By Theorem \ref{ThLiNi}, $G'$ is Hamiltonian unless
$G'\in\mathcal{B}_n^k$, or $k=2$, $n=4$ and $G'=L_1$ or $L_2$. Thus
$G$ is traceable unless $G\in\mathcal{S}_n^k$ or
$G\in\mathcal{T}_n^{k-1}$, or $n=4$, $k=2$ and $G=\emph{\L}$. But if
$G\in\mathcal{T}_n^{k-1}$, or $n=4$, $k=2$ and $G=\emph{\L}$, then
$\delta(G)\leq k-1$, a contradiction. Thus $G\in\mathcal{S}_n^k$.

Now assume that $k=0$. Then
$\rho(\widehat{G})\leq\rho(\widehat{T_n^0})=\rho(\widehat{S_n^1})$.
If $G$ has no isolated vertex, then $\delta(G)\geq 1$ and by the
above analysis, $G$ is traceable unless $G\in\mathcal{S}_n^1$. If
$G$ has an isolated vertex, then $G\subseteq\varGamma_n^0$ or
$G\subseteq T_n^0$. But if $G\subseteq\varGamma_n^0$ or $G\varsubsetneq
T_n^0$, then $\rho(\widehat{G})>\rho(\widehat{T_n^0})$, a contradiction. Thus
$G\in\mathcal{S}_n^1\cup\{T_n^0\}$.
{\hfill$\Box$}

\vskip\baselineskip\noindent\textbf{Proof of Theorem \ref{ThqNBGC}.}
Let $G'$ be defined as in the proof of Theorem \ref{ThrhoNBG}. Note
that $\delta(G')\geq\delta(G)\geq k$,
$q(\widehat{G'})=q(\widehat{G})\leq n$. By Theorem \ref{ThLiNi'},
$G'$ is Hamiltonian unless $G'\in\bigcup_{k=1}^{\lfloor
n/2\rfloor}\mathcal{B}_n^k$, or $n=4$ and $G'=L_1$ or $L_2$. Thus
$G$ is traceable unless $G\in(\bigcup_{k=1}^{\lfloor
(n-1)/2\rfloor}\mathcal{S}_n^k)\cup(\bigcup_{k=0}^{\lfloor
n/2\rfloor-1}\mathcal{T}_n^k)$, or $n=4$ and $G=\emph{\L}$. The
proof is complete.
{\hfill$\Box$}
\section{Concluding remarks}
In fact, during our proofs of main theorems, we have actually proved the following theorems.
All these results maybe stimulate our further study.
\begin{theorem}\label{Th6.1}
Let $G$ be a balanced bipartite graph on $2n$ vertices, with minimum
degree $\delta(G)\geq k$, where $k\geq 1$ and $n\geq(k+2)^2$.
If $\rho(G)\geq\sqrt{n(n-k-1)}$,
then $G$ is traceable unless $G\subseteq
Q_n^k$ or $k=1$ and $G\subseteq R_n^1$.
\end{theorem}

\begin{theorem}\label{Th6.2}
Let $G$ be a balanced bipartite graph on $2n$ vertices, with minimum
degree $\delta(G)\geq k$, where $k\geq 1$ and $n\geq(k+2)^2$.
If $q(G)\geq 2n-k-1$, then
$G$ is traceable unless $G\subseteq
Q_n^k$ or $k=1$ and $G\subseteq R_n^1$.
\end{theorem}

\begin{theorem}\label{Th6.3}
Let $G$ be a nearly balanced bipartite graph on $2n-1$ vertices, with
minimum degree $\delta(G)\geq k$, where $k\geq 1$ and $n\geq(k+1)^2$.
If $\rho(G)\geq\sqrt{n(n-k-1)}$, then $G$ is traceable unless
$G\subseteq S_n^k$.
\end{theorem}

\begin{theorem}\label{Th6.4}
Let $G$ be a nearly balanced bipartite graph on $2n-1$ vertices, with
minimum degree $\delta(G)\geq k$, where $k\geq 1$ and $n\geq(k+1)^2$.
If $q(G)\geq 2n-k-1$, then $G$ is traceable unless
$G\subseteq S_n^k$.
\end{theorem}
On the other hand, notice that in Theorem \ref{ThMoMo}, the order of a graph is required to be
linear multiple of the minimum degree of a graph. But in our Theorems \ref{ThrhoBG},
\ref{ThqBG}, \ref{ThrhoNBG} and \ref{ThqNBG}, the order of a graph is required
to be at least square multiple of minimum degree of a graph. It is natural
to ask whether the required order could be improved to linear multiple of minimum degree
of the graph. Till now, we cannot solve this problem.

\section*{Acknowledgment}
B.-L. Li was supported by the National Natural Science Foundation of China (No.\ 11601429).
B. Ning was supported by the National Natural Science Foundation of China (No.\ 11601379).
The authors are very grateful to the referee whose suggestions largely
improve the presentation of this paper.

\end{document}